\documentclass[12pt]{amsart}[2013/03/28]
\usepackage[utf8x]{inputenc}

\usepackage[a4paper]{geometry}
\usepackage{tikz}
\tikzset{x=0.1em,y=0.1em}

\newtheoremstyle{mythm}
  {9pt}
  {9pt}
  {\slshape}
  {0pt}
  {\bfseries}
  {.}
  { }
  {\thmname{#1} \thmnumber{ #2}\thmnote{ (#3)}}

\theoremstyle{mythm} 
\newtheorem{theorem}{Theorem}[section]
\newtheorem{proposition}[theorem]{Proposition}
\theoremstyle{definition} 
\newtheorem{definition}[theorem]{Definition}
\newtheorem{remark}[theorem]{Remark}
\newtheorem{example}[theorem]{Example}
\newcommand{\alg}[1]{\mathcal{#1}}                

\newcommand{\abs}[1]{\left\lvert #1 \right\rvert}
\DeclareMathOperator{\id}{{\mathrm id}}                     
\DeclareMathOperator{\IE}{{\mathbf E}}                     
\DeclareMathOperator{\mass}{{\mathbf P}}                     
\newcommand{\IC}{\mathbf{C}}                     

\newcommand{\NC}{NC}  
\renewcommand{\phi}{\varphi}
\newcommand{\interw}{\tilde\cup}

\title[Free nested cumulants]{Free nested cumulants and an analogue of a formula of Brillinger}
\author{Franz Lehner}
\address{Institut f\"ur mathematische Strukturtheorie\\
Graz University of Technology, Steyrergasse 30, 8010 Graz}
\email{lehner@math.tugraz.at}

\keywords{Multivariate free cumulants, conditioned cumulants, Brillinger's formula}                                     

\subjclass[2000]{Primary: 46L54 ; Secondary: 05A18}

\begin{document}

\date{\today{}}

\maketitle{}

\begin{abstract}
  We prove a free analogue of Brillinger's formula (sometimes called
  ``law of total cumulance'') which expresses
  classical cumulants in terms of conditioned cumulants.
  As expected, the formula is obtained by replacing
  the lattice of set partitions by the lattice of noncrossing set partitions
  and using and an appropriate notion of noncommutative nested products.
  As an application we reprove a characterization of freeness due to
  Nica, Shlyakhtenko and Speicher by M\"obius inversion techniques,
  without recourse to the Fock space model for free random variables.
\end{abstract}

\section{Introduction and Definitions}
Cumulants 
describe the combinatorial aspects of independence.
Various notions of independence give rise to different kinds of cumulants,
see \cite{Lehner:2004:cumulants1} for a general approach.
In the present paper we concentrate on some aspects of
classical and free cumulants.

\subsection{Classical cumulants}
Classical cumulants can be introduced essentially in two different ways,
via the Fourier transform or via M\"obius inversion on the partition lattice.
For our purposes it will be convenient to use the latter approach.
Let us fix some notation first. 
Denote by~$\Pi_n$ the lattice of set partitions of order~$n$
with  refinement order.
For a partition~$\pi=\{\pi_1,\pi_2,\dots,\pi_p\}\in\Pi_n$
denote by~$\abs{\pi}=p$ its \emph{size}.
Let~$(\Omega,\alg{A},\mass)$ be a probability space with expectation
functional~$\IE$, 
then for a finite sequence
of random variables~$X_1$,~$X_2$, \dots,~$X_n$ on~$\Omega$
we define the partitioned moment functional by
$$
m_\pi(X_1,X_2,\dots,X_n) = \prod_j \IE \prod_{i\in \pi_j} X_i
$$
and the cumulants by
$$
\kappa_\pi(X_1,X_2,\dots,X_n)
 = \sum_{\substack{\sigma\in\Pi_n \\ \sigma\le\pi}}
    m_\sigma(X_1,X_2,\dots,X_n)
    \,\tilde{\mu}(\pi,\hat1_n)
$$
where $\tilde{\mu}$ is the M\"obius function on the partition lattice~$\Pi_n$
\cite{Stanley:2012:enumerative}.
Both $m_\pi$ and $\kappa_\pi$ are multilinear functionals.
For $\pi=\hat1_n$ we shall write $\kappa_n$ instead of $\kappa_{\hat1_n}$.
Then  $\kappa_\pi$ also factorizes along the blocks $\pi_j$ of $\pi$,
namely
$$
\kappa_\pi(X_1,X_2,\dots,X_n)
= \prod_j \kappa_{\abs{\pi_j}}(X_i:i\in\pi_j)
$$

The fundamental result of cumulant theory states that \emph{mixed cumulants vanish}.
That is,
if we can divide the random variables $X_1,X_2,\dots,X_n$ into two (nonempty) independent
groups then the cumulant $\kappa_n(X_1,X_2,\dots,X_n)$ vanishes.

An analogous construction can be done for  conditional expectations
with respect to a
sub-$\sigma$-algebra~$\alg{F}\subseteq\alg{A}$, 
by defining the partitioned conditional 
expectations to be the $\alg{F}$-measurable random variables
$$
\IE_\pi(X_1,X_2,\dots,X_n | \alg{F}) = \prod_j \IE (\prod_{i\in \pi_j} X_i | \alg{F} )
$$
and accordingly the \emph{conditioned cumulants} to be the $\alg{F}$-measurable 
random variables 
$$
\kappa_\pi(X_1,X_2,\dots,X_n | \alg{F})
 = \sum_{\sigma\leq\pi}
    \IE_\sigma(X_1,X_2,\dots,X_n | \alg{F})
    \,\tilde{\mu}(\pi,\hat1_n)
$$
The conditioned cumulants are again multiplicative on blocks and can be used to detect
conditional independence, namely if $X_1,X_2,\dots,X_n$ can be divided into two 
groups which are mutually independent conditionally on $\alg{F}$, then the 
cumulant~$\kappa_n(X_1,X_2,\dots,X_n)$ vanishes.

\subsection{Free cumulants}
In this section we review the
noncommutative analogues of the classical notions of independence and cumulants
from the point of view of Voiculescu's free probability.
\begin{definition}[{\cite{Voiculescu:1995:operations}}]
  Let $(\alg{A},\phi)$ be a noncommutative $\alg{B}$-valued probability space;
  i.e., $\alg{A}$ is a unital complex algebra, $\alg{B}\subseteq\alg{A}$
  is a unital subalgebra and $\phi:\alg{A}\to\alg{B}$
  is a conditional expectation.
  Subalgebras $\alg{A}_i$ which contain $\alg{B}$ are called
  \emph{free with amalgamation over $\alg{B}$}
  if
  $$
  \phi(X_1 X_2\dotsm  X_n)=0
  $$
  whenever $X_j\in \alg{A}_{i_j}$, $\phi(X_j)=0$ and $i_j\neq i_{j+1}$ $\forall j$.
  When $\alg{B}=\IC$, we recover the definition of \emph{freeness}.
\end{definition}
Freeness with amalgamation is a noncommutative analogue of conditional independence
known from classical probability theory.
The corresponding cumulants are due to Speicher 
\cite{Speicher:1994:multiplicative,Speicher:1998:combinatorial}.
Roughly speaking, free cumulants are defined by replacing the lattice
of all partitions in the definition of the classical cumulants
by the lattice of \emph{noncrossing partitions}.
See \cite[Prop.~4.17]{Lehner:2004:cumulants1} for an explanation why noncrossing
partitions appear.
\begin{definition}
  A partition $\pi$ is \emph{noncrossing} if there is no sequence
  $i<j<k<l$ s.t.\ $i\sim_\pi k$ and $j\sim_\pi l$ but $i\not\sim_\pi j$.
  The noncrossing partitions of order $n$
  form a lattice which we denote by $\NC_n$.
\end{definition}
Equivalently, noncrossing partitions can also be characterized recursively
by the property
that there is at least one block which is an interval
and after removing such a block the remaining partition is still non-crossing.
This property will be used in the definitions below.

In the rest of this paper, we use standard poset notation, 
cf.~\cite{Stanley:2012:enumerative}.
The \emph{$\zeta$-function} denotes the order indicator function
$$
\zeta(\pi,\rho) =
\begin{cases}
  1 &\pi\leq \rho\\
  0 &\pi \not\leq \rho
\end{cases}
$$
while by $\mu(\pi,\sigma)$  we will denote the M\"obius
function on the lattice of noncrossing partitions,
i.e., the unique function 
satisfying for every $\pi\leq\sigma$ the identity
$$
\sum_{\pi\leq\rho\leq\sigma}
\zeta(\pi,\rho)\,\mu(\rho,\sigma)=\delta(\pi,\sigma)
.
$$

\begin{definition}[{\cite{Speicher:1998:combinatorial}}]
  Define partitioned moment functionals recursively as follows.
  For a noncrossing partition $\pi\in\NC_n$, let $\pi_j=\{k,k+1,\dots,l\}$
  be an interval block, then
  $$
  \phi_\pi(X_1,X_2,\dots,X_n)
  = \phi_{\pi\setminus\{\pi_j\}}(X_1,
                                 X_2,
                                 \dots,
                                 X_{k-1},
                                 \phi(X_k X_{k+1} \dotsm X_l) X_{l+1},
                                 \dots,
                                 X_n)
  $$
  The \emph{free \emph{or} noncrossing cumulants} are defined by M\"obius inversion on $\NC_n$:
  $$
  C_\pi^\phi(X_1,X_2,\dots,X_n)
  = \sum_{\sigma\leq\pi}
     \phi_\sigma(X_1,X_2,\dots,X_n)
     \, \mu(\sigma,\pi)
  $$
  We will also write $C_n^\phi$ for $C_{\hat1_n}^\phi$ and it follows that
  the cumulants are also multiplicative on blocks, that is,
  if $\pi_j=\{k,k+1,\dots,l\}$ is an interval block of $\pi$ of length $m$,
  then
  $$
  C_\pi^\phi(X_1,X_2,\dots,X_n)
  = C_{\pi\setminus\{\pi_j\}}^\phi(X_1,
                                   X_2,
                                   \dots,
                                   X_{k-1},
                                   C_m^\phi(X_k X_{k+1} \dotsm X_l) X_{l+1},
                                   \dots,
                                   X_n)
  $$
  Moreover, the $\alg{B}$-\emph{module property} holds for expectations 
  \begin{align*}
    \phi_\pi(bX_1,\dots,X_nb') &= b\, \phi_\pi(X_1,\dots,X_n)\,b'\\
    \phi_\pi(X_1,\dots,X_{k-1},bX_k,\dots,X_n) &=
       \phi_\pi(X_1,\dots,X_{k-1}b,X_k,\dots,X_n) 
  \end{align*}
  for all $b,b'\in\alg{B}$, 
  as well as for cumulants: 
  \begin{align*}
    C^\phi_\pi(bX_1,\dots,X_nb') &= b\, C^\phi_\pi(X_1,\dots,X_n)\,b'\\
    C^\phi_\pi(X_1,\dots,X_{k-1},bX_k,\dots,X_n) &=
       C^\phi_\pi(X_1,\dots,X_{k-1}b,X_k,\dots,X_n)
  \end{align*}
\end{definition}
Note that for $\alg{B}=\IC$ this simply means that
$$
C_\pi^\phi(X_1,X_2,\dots,X_n)
= \prod_j C_{\abs{\pi_j}}(X_i : i\in\pi_j)
.
$$

The starting point of this paper is the following formula for classical cumulants,  
due to Brillinger \cite{Brillinger:1969:calculation}:
\begin{equation}
  \label{eq:Brillinger}
  \kappa_n(X_1,X_2,\dots,X_n)
  = \sum_{\pi\in\Pi_n}
     \kappa_{\abs{\pi}}( \kappa_{\abs{\pi_j}}(X_i : i\in \pi_j | \alg{B})
                                        : j=1,\dots,\abs{\pi})
\end{equation}
where for a partition~$\pi=\{B_1,B_2,\dots,B_p\}\in\Pi_n$ we
denote by~$\abs{\pi}=p$ its \emph{size}.

We establish an analogue of this formula
for free cumulants by adapting a lattice theoretical proof due to Speed
\cite{Speed:1983:cumulantsI}.
Noncommutativity prevents a direct generalization of \eqref{eq:Brillinger},
therefore we propose \emph{nested cumulants} as a replacement for ``cumulants of cumulants''.
To illustrate this issue we first consider cumulants of products from an abstract point
of view.

\section{Cumulants of nested products}
\label{sec:nested}
We want to define cumulants of products, where the products are not taken in
linear order. To do this, we first give a definition and then discuss its
connection
to cumulants of products.
\begin{definition}
  Let $\rho\leq\sigma$ be two noncrossing partitions of order $n$
  and $X_1$, $X_2$,\ldots{}, $X_n$ be noncommutative random variables.
  Then we define the partial cumulant
  $$
  C_{\rho,\sigma}(X_1,X_2,\dots,X_n)
  = \sum_{\rho\leq \pi\leq \sigma}
    \phi_\pi(X_1,X_2,\dots,X_n)
    \mu(\pi,\sigma)
  .
  $$
\end{definition}
Note that in particular for $\rho=\hat{0}_n$ we obtain
the usual cumulant $C_{\hat{0},\sigma}=C_{\sigma}$, 
while for  $\rho=\sigma$ we get the moment
$C_{\sigma,\sigma}=\phi_{\sigma}$.
For intermediate partitions we get a generalization of cumulants of products.

\begin{definition}
  Let $\rho =\{\rho_1,\rho_2,\dots,\rho_r\}$
  and $\sigma=\{\sigma_1,\sigma_2,\dots,\sigma_s\}$
  be two set partitions such that $\rho\leq\sigma$.
  Here  the blocks are numbered according to their minimal elements.
  Then every block of $\rho$ is contained in some block
  of $\sigma$ and by collapsing the blocks of $\rho$ we
  can define
  $\sigma/\rho=\{\hat{\sigma}_1,\dots,\hat{\sigma}_s\}$
  to be the unique partition of the set $\{1,2,\dots,r\}$
  such that
  $\sigma_i = \bigcup_{j\in\hat{\sigma}_i} \rho_j$ for every $i$.
\end{definition}

\begin{remark}
\label{rem:intervalpartition}
When  $\rho$ is an interval partition,
say $\rho =\{\rho_1,\rho_2,\dots,\rho_r\}$,
where $\rho_1=\{1,2,\dots,n_1\}$, $\rho_2=\{n_1+1,2,\dots,n_2\}$, \ldots,
$\rho_r=\{n_{r-1}+1,2,\dots,n_r=n\}$,
and $\sigma$ is noncrossing,
then $\sigma/\rho$ is noncrossing as well and the partial cumulant
coincides with the cumulant of the products
$$
C_{\rho,\sigma}(X_1,X_2,\dots,X_n)
= C_{\sigma/\rho}(X_1X_2\dotsm X_{n_1}, 
                 X_{n_1+1}\dotsm X_{n_2},
                 \dots,
                 X_{n_{r-1}+1}\dotsm X_{n})
.
$$
\end{remark}

There is a formula for cumulants of products in terms of simple cumulants,
which is due to Leonov and Shiryaev in the classical case
\cite{LeonovShiryaev:1959:method}
and to Speicher and Krawchyk in the free case \cite{KrawczykSpeicher:2000:combinatorics}.
It immediately generalizes to the partial cumulants
(cf.~\cite[Prop.~10.11]{NicaSpeicher:2006:lectures}).
\begin{proposition}
  For partitions $\rho\leq\sigma$ we have
  $$
  C_{\rho,\sigma}(X_1,X_2,\dots,X_n)
  = \sum_{\substack{\tau\\ \tau\vee \rho = \sigma}} 
    C_\tau(X_1,X_2,\dots,X_n)
  $$
\end{proposition}
\begin{proof}
  \begin{align*}
    C_{\rho,\sigma}
    &= \sum_\pi \phi_\pi(X_1,X_2,\dots,X_n)
       \,\zeta(\rho,\pi)\,\mu(\pi,\sigma)
       \\
    &= \sum_\pi\sum_\tau  C_\tau(X_1,X_2,\dots,X_n)
       \,\zeta(\tau,\pi)
       \,\zeta(\rho,\pi)\,\mu(\pi,\sigma)
       \\
    &= \sum_\tau   C_\tau(X_1,X_2,\dots,X_n)
       \sum_\pi
       \,\zeta(\tau\vee \rho,\pi)
       \,\mu(\pi,\sigma)
       \\
    &= \sum_\tau   C_\tau(X_1,X_2,\dots,X_n)
       \,
       \delta(\tau\vee \rho,\sigma)
  \end{align*}
\end{proof}

\begin{remark}
  The procedure presented in this section
  can also be carried out for classical cumulants, i.e.,
  on the full partition lattice,
  however because of commutativity it simply leads to a rearrangement of
  cumulants of products, namely
  $$
  \kappa_{\rho,\sigma}(X_1,X_2,\dots,X_n)
  =  \kappa_{\sigma/\rho}(\prod_{i\in b} X_i : b \in \rho)
  $$
\end{remark}

\section{Conditioned free cumulants}

Suppose we are given algebras $\alg{C}\subseteq \alg{B}\subseteq\alg{A}$
and conditional expectations $\alg{A}\xrightarrow{\psi}\alg{B}\xrightarrow{\phi}\alg{C}$.
We identify $\phi$ with $\phi\circ\psi:\alg{A}\to\alg{C}$
and wish to express the $\alg{C}$-valued cumulants $C^\phi$ in terms of the 
$\alg{B}$-valued cumulants $C^\psi$.
The next definition is rather formal and should be read 
with the examples following it at hand.
\begin{definition}
  \label{def:condfree}
  We define a partitioned moment function $\phi$ of the partitioned cumulants
  $C_\pi^\psi$, namely for $\sigma\geq\pi$ we define 
  $\phi_\sigma\circ C_\pi^\psi(X_1,X_2,\dots,X_n)$ recursively as follows.
  Let $\sigma_j=\{k+1,\dots,l\}$ be an interval block of $\sigma$
  and $\pi|_{\sigma_j}=\{\pi_{i_1},\pi_{i_2},\dots,\pi_{i_m}\}$ the
  blocks of $\pi$ which are contained in $\sigma_j$,
  then we put
  \begin{multline*}
  \phi_\sigma\circ \psi_\pi(X_1,X_2,\dots,X_n)\\
  = \phi_{\sigma\setminus\{\sigma_j\}}\circ \psi_{\pi\setminus \pi|_{\sigma_j} }
      (X_1,
       X_2,
       \dots,
       X_k,
       \phi( \psi_{\pi|_{\sigma_j}}(X_{k+1},\dots,X_l)) X_{l+1},
       X_{l+2},
       \dots, 
       X_n
       )
  \end{multline*}
  and
  $$
  \phi_\sigma\circ C_\pi^\psi(X_1,X_2,\dots,X_n)
  = \sum_{\tau\leq\pi}
     \phi_\sigma\circ\psi_\tau(X_1,X_2,\dots,X_n)\,\mu(\tau,\pi)
  .
  $$
  By multiplicativity we have
  \begin{multline*}
  \phi_\sigma\circ C_\pi^\psi(X_1,X_2,\dots,X_n)\\
  = \phi_{\sigma\setminus\{\sigma_j\}}\circ C^\psi_{\pi\setminus \pi|_{\sigma_j} }
      (X_1,
       X_2,
       \dots,
       X_k,
       \phi( C^\psi_{\pi|_{\sigma_j}}(X_{k+1},\dots,X_l)) X_{l+1},
       X_{l+2},
       \dots, 
       X_n
       )
    .
  \end{multline*}
  Moreover the M{\"o}bius inversion principle and the invariance
  $\phi=\phi\circ\psi$ imply a generalized moment-cumulant formula
  $$
  \phi_\sigma(X_1,X_2,\dots,X_n)
  = \sum_{\pi\leq\sigma} \phi_\sigma\circ C_\pi^\psi(X_1,X_2,\dots,X_n)
  .
  $$
  Now we apply the cumulant construction in each block of $\sigma$
  to define ``cumulants of cumulants'' or \emph{nested cumulants}:
  $$
  C_\sigma^\phi\circ C_\pi^\psi(X_1,X_2,\dots,X_n)
  = \sum_{\pi\leq\rho\leq\sigma} \phi_\rho\circ C_\pi^\psi(X_1,X_2,\dots,X_n)\,\mu(\rho,\sigma)
  \,.
  $$
  
  In total this means that
  $$
  C^\phi_\sigma\circ C^\psi_\pi(X_1,X_2,\dots,X_n)
  =\sum_{\tau\leq\pi} \sum_{\pi\leq\rho\leq\sigma}
    \phi_\rho\circ\psi_{\tau}(X_1,X_2,\dots,X_n)
    \,\mu(\rho,\sigma)
    \,\mu(\tau,\pi)
  $$
  This function is multiplicative on the blocks and we have by M\"obius inversion
  $$
  \phi_\sigma\circ C_\pi^\psi(X_1,X_2,\dots,X_n)
  = \sum_{\pi\leq\rho\leq\sigma}
     C_\rho^\phi\circ C_\pi^\psi(X_1,X_2,\dots,X_n)
  $$
\end{definition}

\begin{example}
\label{ex:1}
Again, if $\rho$ is an interval partition as in 
Remark~\ref{rem:intervalpartition}
then we get the analogous formula
\begin{multline}
  \label{eq:CsigmaoCrhointerval}
C_\sigma^\phi\circ C_\rho^\psi(X_1,X_2,\dots,X_n)
\\
= C_{\sigma/\rho}^\phi(C_{n_1}(X_1,X_2,\dots, X_{n_1}), 
                  C_{n_2-n_1}(X_{n_1+1},\dots, X_{n_2}),
                  \dots,
                  C_{n_r-n_{r-1}}(X_{n_{r-1}+1}\dotsm X_{n}))
.
\end{multline}
\end{example}
\begin{example}
\label{ex:2}  
If $\rho$ is not an interval partition
then the nested cumulant becomes more complicated.
As an example consider $\pi=
\begin{tikzpicture}
  \draw (2,0)--(2,7.5);
  \draw (8,0)--(8,7.5);
  \draw (14,0)--(14,4.5);
  \draw (20,0)--(20,4.5);
  \draw (26,0)--(26,4.5);
  \draw (32,0)--(32,4.5);
  \draw (38,0)--(38,7.5);
  \draw (44,0)--(44,7.5);
  \draw (14,4.5)--(20,4.5);
  \draw (26,4.5)--(32,4.5);
  \draw (2,7.5)--(44,7.5);
\end{tikzpicture}$ and $\sigma=
\begin{tikzpicture}
  \draw (2,0)--(2,7.5);
  \draw (8,0)--(8,7.5);
  \draw (14,0)--(14,4.5);
  \draw (20,0)--(20,4.5);
  \draw (26,0)--(26,4.5);
  \draw (32,0)--(32,4.5);
  \draw (38,0)--(38,7.5);
  \draw (44,0)--(44,7.5);
  \draw (14,4.5)--(32,4.5);
  \draw (2,7.5)--(44,7.5);
\end{tikzpicture}$,
then
\begin{align*}
  \psi_\pi(X_1,X_2,\dots,X_8) 
  &= \psi(X_1 X_2\, \psi(X_3 X_4)\, \psi(X_5 X_6)\, X_7 X_8)
  \\
  \phi_\sigma\circ\psi_\pi(X_1,X_2,\dots,X_8)
  &= \phi(\psi(X_1 X_2\, \phi(\psi(X_3 X_4)\, \psi(X_5 X_6))\, X_7 X_8))
  \\
  \phi_\sigma\circ C^\psi_\pi(X_1,X_2,\dots,X_8)
  &= \phi(C_4^\psi(X_1, X_2, \phi(C_2^\psi(X_3, X_4)\,C_2^\psi(X_5, X_6))\, X_7, X_8))
  \\
  C^\phi_\sigma\circ C^\psi_\pi(X_1,X_2,\dots,X_8)
  &= \phi(C_4^\psi(X_1, X_2, C_2^\phi(C_2^\psi(X_3, X_4),C_2^\psi(X_5, X_6))\, X_7, X_8))
\end{align*}

\end{example}
\begin{example}
  \label{ex:3}
  The previous examples might give the impression that the conditioned 
  cumulants can always be expressed in terms of the $\psi$-cumulants.
  Here is a nontrivial example which shows that this is not the case.
  \begin{align*}
    C_3^\phi\circ C_{
\begin{picture}(14,6.5)(1,0)
  \put(2,0){\line(0,1){7.5}}
  \put(8,0){\line(0,1){4.5}}
  \put(14,0){\line(0,1){7.5}}
  \put(8,4.5){\line(1,0){0}}
  \put(2,7.5){\line(1,0){12}}
\end{picture}
}(X_1, X_2, X_3)
  &= \phi_{
\begin{picture}(14,6.5)(1,0)
  \put(2,0){\line(0,1){7.5}}
  \put(8,0){\line(0,1){4.5}}
  \put(14,0){\line(0,1){7.5}}
  \put(8,4.5){\line(1,0){0}}
  \put(2,7.5){\line(1,0){12}}
\end{picture}
}\circ C^\psi_{
\begin{picture}(14,6.5)(1,0)
  \put(2,0){\line(0,1){7.5}}
  \put(8,0){\line(0,1){4.5}}
  \put(14,0){\line(0,1){7.5}}
  \put(8,4.5){\line(1,0){0}}
  \put(2,7.5){\line(1,0){12}}
\end{picture}
}(X_1,X_2,X_3) \mu(
\begin{picture}(14,6.5)(1,0)
  \put(2,0){\line(0,1){7.5}}
  \put(8,0){\line(0,1){4.5}}
  \put(14,0){\line(0,1){7.5}}
  \put(8,4.5){\line(1,0){0}}
  \put(2,7.5){\line(1,0){12}}
\end{picture}
,
\begin{picture}(14,3.5)(1,0)
  \put(2,0){\line(0,1){4.5}}
  \put(8,0){\line(0,1){4.5}}
  \put(14,0){\line(0,1){4.5}}
  \put(2,4.5){\line(1,0){12}}
\end{picture}
)
\\
& \qquad +
\phi_{
\begin{picture}(14,3.5)(1,0)
  \put(2,0){\line(0,1){4.5}}
  \put(8,0){\line(0,1){4.5}}
  \put(14,0){\line(0,1){4.5}}
  \put(2,4.5){\line(1,0){12}}
\end{picture}
}\circ
C_{
\begin{picture}(14,6.5)(1,0)
  \put(2,0){\line(0,1){7.5}}
  \put(8,0){\line(0,1){4.5}}
  \put(14,0){\line(0,1){7.5}}
  \put(8,4.5){\line(1,0){0}}
  \put(2,7.5){\line(1,0){12}}
\end{picture}
}(X_1,X_2,X_3)
\,
\mu(
\begin{picture}(14,3.5)(1,0)
  \put(2,0){\line(0,1){4.5}}
  \put(8,0){\line(0,1){4.5}}
  \put(14,0){\line(0,1){4.5}}
  \put(2,4.5){\line(1,0){12}}
\end{picture}
,
\begin{picture}(14,3.5)(1,0)
  \put(2,0){\line(0,1){4.5}}
  \put(8,0){\line(0,1){4.5}}
  \put(14,0){\line(0,1){4.5}}
  \put(2,4.5){\line(1,0){12}}
\end{picture}
)
\\
&= \phi( C^\psi_{
\begin{picture}(14,6.5)(1,0)
  \put(2,0){\line(0,1){7.5}}
  \put(8,0){\line(0,1){4.5}}
  \put(14,0){\line(0,1){7.5}}
  \put(8,4.5){\line(1,0){0}}
  \put(2,7.5){\line(1,0){12}}
\end{picture}
}(X_1, X_2, X_3))
- \phi_{
\begin{picture}(14,6.5)(1,0)
  \put(2,0){\line(0,1){7.5}}
  \put(8,0){\line(0,1){4.5}}
  \put(14,0){\line(0,1){7.5}}
  \put(8,4.5){\line(1,0){0}}
  \put(2,7.5){\line(1,0){12}}
\end{picture}
}(
C^\psi_{
\begin{picture}(14,6.5)(1,0)
  \put(2,0){\line(0,1){7.5}}
  \put(8,0){\line(0,1){4.5}}
  \put(14,0){\line(0,1){7.5}}
  \put(8,4.5){\line(1,0){0}}
  \put(2,7.5){\line(1,0){12}}
\end{picture}
}(X_1, X_2, X_3))
\\
&= \phi( C^\psi_2(X_1,\psi(X_2)X_3))
   - \phi(C^\psi_2(X_1,\phi(X_2)X_3))
\\
&= \phi( C^\psi_2(X_1, (\psi(X_2)-\phi(X_2))X_3))
  \end{align*}
\end{example}

\begin{example}
  Here is an example exhibiting some partial commutativity.
  Let $(\alg{A},\phi)$
  and $(\alg{B},\psi)$ be two noncommutative probability spaces.
  For the sake of simplicity assume that both $\phi$ and $\psi$ are
  $\IC$-valued  expectations.
  Consider the inclusions $\IC\subseteq \alg{B}\simeq I\otimes \alg{B}
  \subseteq \alg{A}\otimes\alg{B}$ and the corresponding expectations
  $\tilde\phi = \phi\otimes\id : \alg{A}\otimes\alg{B} \to \alg{B}$
  and $\psi:\alg{B}\to\IC$.
  Note that if $\alg{A}_i$ are free subalgebras of a noncommutative probability
  space,
  then   $\alg{A}_i\otimes\alg{B}$ are free with amalgamation over $\alg{B}$
  in $\alg{A}\otimes\alg{B}$.
  Then for any sequence of  simple tensors $a_1\otimes b_1$,  $a_2\otimes b_2$,
  \ldots{} $a_n\otimes b_n$
  the nested expectations and cumulants as defined above are
  \begin{align*}
    \psi_\sigma\circ \tilde{\phi}_\pi(a_1\otimes b_1, a_2\otimes b_2,\dots a_n\otimes b_n) 
       &= \phi_\sigma(a_1,a_2,\dots,a_n) \, \psi_\pi(b_1,b_2,\dots,b_n)\\
    \psi_\sigma\circ C^{\tilde{\phi}}_\pi(a_1\otimes b_1, a_2\otimes b_2,\dots a_n\otimes b_n) 
       &= \phi_\sigma(a_1,a_2,\dots,a_n) \, C^\psi_\pi(b_1,b_2,\dots,b_n)\\
    C^\psi_\sigma\circ C^{\tilde{\phi}}_\pi(a_1\otimes b_1, a_2\otimes b_2,\dots a_n\otimes b_n) 
       &= C^\phi_\sigma(a_1,a_2,\dots,a_n) \, C^\psi_\pi(b_1,b_2,\dots,b_n)\\
  \end{align*}
\end{example}

\begin{remark}
  Note that if we apply this definition with classical instead of free
  cumulants, the analogue of \eqref{eq:CsigmaoCrhointerval} holds for arbitrary
  partitions. Indeed, denote by $\IE^{\alg{F}}$ and $\kappa^{\alg{F}}$ the
  conditional expectations and cumulants with respect to a $\sigma$-subfield
  $\alg{F}$ of the given probability space. Then we define for a pair
  of set partitions $\sigma\geq \pi$ the partitioned expectations and cumulants
  as before, replacing noncrossing partitions by arbitrary partitions and obtain
  \begin{align*}
  \IE_\sigma\circ \IE^{\alg{F}}(X_1,X_2,\dots,X_n)
  &= \prod_{c\in\sigma} 
       \IE\prod_{\substack{b\in\pi \\ b\subseteq c}} \IE[\prod_{i\in b} X_i  | \alg{F}]\\
  \IE_\sigma\circ \kappa^{\alg{F}}(X_1,X_2,\dots,X_n)  
  &= \sum_{\tau\leq\pi} \IE_\sigma\circ\IE^{\alg{F}}_\tau(X_1,X_2,\dots,X_n)
  \,\mu(\tau,\pi)\\
  &= \prod _{c\in\sigma} 
       \IE\prod_{\substack{b\in\pi \\ b\subseteq c}} \kappa^{\alg{F}}(X_i:i\in b)\\
  \kappa_\sigma\circ \kappa^{\alg{F}}_\pi(X_1,X_2,\dots,X_n) 
  &= \kappa_{\sigma/\rho}(\kappa^{\alg{F}}_{\abs{b}}(X_i:i\in b): b\in \pi)
  \end{align*}
  where $\sigma/\rho$ is the partition obtained from $\sigma$ by collapsing
  each block of $\pi$ to a singleton as defined in section~\ref{sec:nested},
  which implies that the intervals
  $[\pi,\sigma]$ and $[\hat{0}_m,\sigma/\rho]$ are isomorphic as posets.

\end{remark}

Here is now the analogue of Brillinger's formula \eqref{eq:Brillinger} 
for free cumulants.
As expected, noncrossing partitions appear, but we also have to take
care of noncommutativity.
\begin{theorem}
  \label{thm:ConditionedCumulants}
  $$
  C_n^\phi(X_1,X_2,\dots,X_n)
  = \sum_{\sigma\in \NC_n}
     C_n^\phi\circ C_\sigma^\psi(X_1,X_2,\dots,X_n)
  $$
\end{theorem}
\begin{proof}
  The proof of \cite{Speed:1983:cumulantsI} can be repeated literally
  after replacing the lattice $\Pi_n$ by its sublattice $\NC_n$:
  \begin{align*}
    C_n^\phi(X_1,X_2,\dots,X_n)
    &= \sum_{\pi\in \NC_n}
        \phi_\pi(X_1,X_2,\dots,X_n)\,\mu(\pi,\hat1_n) \\
    &= \sum_{\pi\in \NC_n}
        \sum_{\sigma\leq\pi}
         \sum_{\sigma\leq\rho\leq\pi}
          C_\rho^\phi\circ C_\sigma^\psi(X_1,X_2,\dots,X_n)
          \, \mu(\pi,\hat1_n)
    \\
    &= \sum_{\pi\in \NC_n}
        \sum_{\rho\in\NC_n}
         \sum_{\sigma\in\NC_n}
          C_\rho^\phi\circ C_\sigma^\psi(X_1,X_2,\dots,X_n)
          \, \zeta(\sigma,\rho)
          \, \zeta(\rho,\pi)
          \, \mu(\pi,\hat1_n)
    \\
    &= \sum_{\rho\in\NC_n}
        \sum_{\sigma\leq\rho}
         C_\rho^\phi\circ C_\sigma^\psi(X_1,X_2,\dots,X_n)
          \, \delta(\rho,\hat1_n)
  \end{align*}
\end{proof}

\section{An application}

As an application we reprove a characterization of freeness from the recent paper
\cite{NicaShlyakhtenkoSpeicher:2002:operatorI}.
To illustrate our approach, let us first give a proof of a more or less trivial formula
from the latter paper.
\begin{proposition}[{\cite[Theorem~3.1]{NicaShlyakhtenkoSpeicher:2002:operatorI}}]
  \label{nss:ckpsi}
  Let $\alg{C}\subseteq\alg{B}\subseteq\alg{A}$ and $\psi:\alg{A}\to\alg{B}$,
  $\phi:\alg{A}\to\alg{C}$ be as before.
  If the $\psi$-valued cumulants of $X_1,X_2,\dots,X_n$ satisfy
  $$
  C_k^\psi(X_{i_1} c_1,X_{i_2} c_2,\dots,X_{i_{k-1}} c_{k-1},X_{i_k})
  \in \alg{C}
  $$
  for all choices of indices $i_1,i_2,\dots,i_k$ and elements $c_1,\dots,c_{k-1}\in\alg{C}$,
  then actually
  $$
  C_k^\psi(X_{i_1} c_1,X_{i_2} c_2,\dots,X_{i_{k-1}} c_{k-1},X_{i_k})
  = C_k^\phi(X_{i_1} c_1,X_{i_2} c_2,\dots,X_{i_{k-1}} c_{k-1},X_{i_k})
  $$
\end{proposition}
\begin{proof}
  By Theorem~\ref{thm:ConditionedCumulants} we can expand the
  $\phi$-cumulant in terms of the $\psi$-cumulants
  $$
    C_n^\phi(X_{i_1}c_1,X_{i_2} c_2,\dots,X_{i_{k-1}} c_{k-1},X_{i_k})
    = \sum_\pi
       C_n^\phi\circ C_\pi^\psi(X_{i_1} c_1,
                                          X_{i_2} c_2,
                                          \dots,
                                          X_{i_{k-1}} c_{k-1},
                                          X_{i_k})
  .
  $$
  Now by definition
  \begin{multline*}
    C_n^\phi\circ C_\pi^\psi(X_{i_1} c_1,
                                       X_{i_2} c_2,
                                       \dots,
                                       X_{i_{k-1}} c_{k-1},
                                       X_{i_k})\\
    = \sum_{\sigma\geq\pi}
        \phi_{\sigma}\circ C_\pi^\psi(X_{i_1} c_1,
                                       X_{i_2} c_2,
                                       \dots,
                                       X_{i_{k-1}} c_{k-1},
                                       X_{i_k})
        \, \mu(\sigma,\hat1_n)
  \end{multline*}
  and by assumption, 
  $$
  \phi_{\sigma}\circ  C_\pi^\psi(X_{i_1} c_1,
                                       X_{i_2} c_2,
                                       \dots,
                                       X_{i_{k-1}} c_{k-1},
                                       X_{i_k})
  = C_\pi^\psi(X_{i_1} c_1,
               X_{i_2} c_2,
               \dots,
               X_{i_{k-1}} c_{k-1},
               X_{i_k})
  $$
  for all $\sigma\geq \pi$
  and $\sum_{\sigma\geq\pi} \mu(\sigma,\hat1_n)=0$ unless $\pi=\hat1_n$.
  Therefore only the summand corresponding to $\pi=\hat1_n$ is nonzero.
\end{proof}

For the final application we need to recall the basic properties
of the Kreweras complement.
\begin{definition}[\cite{Kreweras:1972:partitions}]
  Given two set partitions $\pi$ and $\sigma$ of the same order $n$,
  we denote by $\pi\interw\sigma$ 
  their \emph{interweaved union},
  i.e., the partition of order $2n$ obtained by alternatingly arranging the
  points of $\pi$ and $\sigma$.

  The \emph{Kreweras complement} of a partition $\pi\in\NC_n$
  is defined as the unique  maximal partition $\sigma\in\NC_n$
  s.t.\ $\pi\interw\sigma$ is noncrossing.
\end{definition}
The Kreweras complement is in fact an anti-automorphism of $\NC_n$
which immediately implies the following proposition;
let us however give another proof here by constructing
and explicit bijection to which we will refer later.

\begin{proposition}
  \label{prop:kreweras}
  Let $\pi\in\NC_n$, then
  the intervals $[\,\pi,\hat1_n\,]$ and $[\,0,K(\pi)\,]$ 
  are antiisomorphic via the Kreweras complement. 
  \end{proposition}

  \begin{proof}

  Draw $\pi$ and all the points of $K(\pi)$ between the points of $\pi$.
  Every $\sigma\geq\pi$ is obtained from $\pi$
  by connecting some of its blocks.
  To every possible connection there corresponds a unique connection of two
  points of $K(\pi)$, as follows.
  There are two possible relative positions of two blocks of $\pi$:
  \begin{enumerate}
   \item 
    \begin{picture}(90,15)(1,0)
      \put(10,0){\line(0,1){8.4}}
      \put(20,0){\line(0,1){8.4}}
      \put(30,0){\line(0,1){8.4}}
      \put(43,0){$\cdots$}
      %
      \put(70,0){\line(0,1){8.4}}
      \put(80,0){\line(0,1){8.4}}
      \put(90,0){\line(0,1){8.4}}
      \put(103,0){$\cdots$}
      \put(10,8.4){\line(1,0){20}}
      \put(70,8.4){\line(1,0){20}}
      %
      %
      \put(15,0.3){\circle{0.2}}
      \put(25,0.3){\circle{0.2}}
      \put(35,0.3){\circle{0.2}}
      \put(31.85,0.7){\tiny$\times$}
      \put(65,0.3){\circle{0.2}}
      \put(75,0.3){\circle{0.2}}
      \put(85,0.3){\circle{0.2}}
      \put(95,0.3){\circle{0.2}}
      \put(91.85,0.7){\tiny$\times$}
    \end{picture}
   \item 
     \begin{picture}(160,20)(1,0)
      \put(10,0){\line(0,1){14.4}}
      \put(20,0){\line(0,1){14.4}}
      \put(30,0){\line(0,1){14.4}}
      \put(43,0){$\cdots$}
      %
      \put(70,0){\line(0,1){8.4}}
      \put(80,0){\line(0,1){8.4}}
      \put(90,0){\line(0,1){8.4}}
      \put(100,0){\line(0,1){8.4}}
      \put(113,0){$\cdots$}
      %
      \put(140,0){\line(0,1){14.4}}
      \put(150,0){\line(0,1){14.4}}
      \put(160,0){\line(0,1){14.4}}
      \put(10,14.4){\line(1,0){150}}
      \put(70,8.4){\line(1,0){30}}
      %
      %
      \put(15,0.3){\circle{0.2}}
      \put(25,0.3){\circle{0.2}}
      \put(35,0.3){\circle{0.2}}
      \put(31.85,0.7){\tiny$\times$}
      \put(65,0.3){\circle{0.2}}
      \put(75,0.3){\circle{0.2}}
      \put(85,0.3){\circle{0.2}}
      \put(95,0.3){\circle{0.2}}
      \put(105,0.3){\circle{0.2}}
      \put(101.85,0.7){\tiny$\times$}
      \put(135,0.3){\circle{0.2}}
      \put(145,0.3){\circle{0.2}}
      \put(155,0.3){\circle{0.2}}
      \put(165,0.3){\circle{0.2}}
    \end{picture}
  \end{enumerate}
  In both cases connecting the two blocks of $\pi$ corresponds to connecting
  the points marked with ``$\times$'' in the Kreweras complement.
  \end{proof}

The Kreweras naturally appears in the incidence algebra convolution product
which implements  \emph{multiplicative free convolution}
on the level of cumulants.

\begin{proposition}[\cite{NicaSpeicher:2006:lectures}]
  \label{nicaspeicher:multiplicative} 
  Let $(\alg{A},\psi)$ be a $\alg{B}$-valued probability space
  and
  let $a_1,a_2,\dots,a_n$ and $b_1,b_2,\dots,b_n$ be random variables
  free over $\alg{B}$.
  Then the cumulants of the product are
  $$
  C^\psi_n(a_1b_1, a_2b_2,\dots,a_nb_n)
  =\sum_{\pi\in\NC_n}
  C^\psi_{\pi\interw K(\pi)}(a_1,b_2,a_2,b_2,\dots,a_n,b_n)
  $$
\end{proposition}

With these preparations we are able to provide an alternative proof of the following theorem.

\begin{theorem}[{\cite[Theorem~3.6]{NicaShlyakhtenkoSpeicher:2002:operatorI}}]
  \label{thm:CondCum:NSSthm3.6}
  Let $\alg{C}\subseteq \alg{B}\subseteq\alg{A}$ 
  and $\psi:\alg{A}\to\alg{B}$, $\phi:\alg{A}\to\alg{C}$ as before.
  Let $\alg{C}\subseteq \alg{N}\subseteq\alg{A}$ be another subalgebra
  and
  assume in addition that $\phi:\alg{B}\to\alg{C}$ is faithful.
  Then $\alg{N}$ is free from $\alg{B}$ over $\alg{C}$ if and only if
  for all finite sequences~$X_i\in\alg{N}$ and for all~$b_i\in\alg{B}$
  the identity
  \begin{multline}
  C_n^\psi(X_1b_1,X_2b_2,\dots,X_{n-1}b_{n-1},X_n) \\
  = \phi(C_n^\psi(X_1\phi(b_1),X_2\phi(b_2),\dots,X_{n-1}\phi(b_{n-1}),X_n))
  \end{multline}
  holds. 
  By Proposition~\ref{nss:ckpsi}
  this is equivalent to the statement that
  for all finite sequences~$X_i\in\alg{N}$ and for all~$b_i\in\alg{B}$
  we have
  \begin{multline}
  C_n^\psi(X_1b_1,X_2b_2,\dots,X_{n-1}b_{n-1},X_n)
  \\
  = C_n^\phi(X_1\phi(b_1),X_2\phi(b_2),\dots,X_{n-1}\phi(b_{n-1}),X_n)
  .
  \end{multline}
\end{theorem}
\begin{proof}
  Assume that the factorization formula holds.
  Let $X_1,X_2,\dots,X_n\in\alg{N}$,
  $b_0,b_1,\dots,b_{n}\in \alg{B}$ s.t.\ $\phi(X_i)=0$ and $\phi(b_i)=0$
  (or $b_0=1$ or $b_n=1$ is also allowed).
  We must show that 
  $\phi(b_0 X_1 b_1 X_2 \dotsm X_n b_n)=0$.
  To this end we expand the expectation into $\psi$-cumulants
  \begin{align*}
  \phi(b_0 X_1 b_1\dotsm X_n b_n) 
  &=
  \phi(\psi(b_0 X_1 b_1\dotsm X_n b_n))
  \\
  &= \sum_{\pi\in\NC_n}
     \phi(C_\pi^\psi(b_0X_1b_1,X_2b_2,\dots, X_n b_n))
  \end{align*}
  and
  $C_\pi^\psi(b_0X_1b_1,X_2b_2,\dots, X_n b_n))=0$ for each~$\pi$
  because each $\pi$ has a block which is an interval say of length~$m$
  starting at some~$k$
  and the corresponding cumulant contributes the factor
  $$
  C_m^\psi(X_k b_k, X_{k+1} b_{k+1},\dots X_l)
  = \phi(C_m^\psi(X_k\phi(b_k), X_{k+1} \phi(b_{k+1}),\dots X_l))
  $$
  which vanishes:
  if $m\geq2$ then there is a factor $\phi(b_k)=0$ and if $m=1$,
  then the term is simply $C_1^\psi(X_k)=\phi(C_1^\psi(X_k))=\phi(X_k)=0$.
  Note that we did not need faithfulness of~$\phi$ for this implication.

  For the converse we could use the same argument as in
  \cite{NicaShlyakhtenkoSpeicher:2002:operatorI}, in which a reference algebra~$\alg{N'}$
  is constructed which is also free from~$\alg{B}$ over~$\alg{C}$ and which
  satisfies the cumulant factorization condition and has the same distribution as~ $\alg{N}$.
  It then follows that~$\alg{N}$ satisfies the cumulant factorization condition as well.

  Alternatively, here is a sketch of a direct proof
  using conditioned cumulants.
  By faithfulness it suffices to prove that
  for all finite sequences of random variables~$X_i\in\alg{N}$
  and~$b_i\in\alg{B}$ we have the identity
  \begin{multline*}
  \phi( C_n^\psi(X_1 b_1,
                 X_2 b_2,
                 \dots,
                 X_{n-1} b_{n-1},
                 X_n)b_n) \\
  = \phi(C_n^\psi(X_1 \phi(b_1),
                 X_2 \phi(b_2),
                 \dots,
                 X_{n-1} \phi(b_{n-1}),
                 X_n) b_n)
  \end{multline*}
  and moreover this is equal to
  $$
  C_n^\phi(X_1 \phi(b_1),
           X_2 \phi(b_2),
           \dots,
           X_{n-1} \phi(b_{n-1}),
           X_n \phi(b_n))
  .
  $$
  We proceed by induction and compare the following two formulae
  for 
  $C^\phi_n(X_1 b_1,
             X_2 b_2,
             \dots,
             X_n b_n)
             .
  $
  On the one hand, by freeness we may apply the formula for multiplicative
  convolution from Proposition~\ref{nicaspeicher:multiplicative}
  \begin{align*}
    C_n^\phi&(X_1 b_1,
             X_2 b_2,
             \dots,
             X_n b_n) \\
    &= \sum_{\pi\in\NC_n}
        C_{\pi\interw K(\pi)}^\phi(X_1,b_1,
                                X_2, b_2,
                                \dots,
                                X_n, b_n) \\
    &= \underbrace{C_{\hat1_n\interw\hat0_n}^\phi(X_1,b_1,
                                   X_2, b_2,
                                   \dots,
                                   X_n, b_n)}_{\displaystyle%
                   C_n^\phi(X_1 \phi(b_1),\dots,X_n\phi(b_n))}
       +
       \sum_{\pi<\hat1_n}
        C_{\pi\interw K(\pi)}^\phi(X_1,b_1,
                                X_2, b_2,
                                \dots,
                                X_n, b_n) \\
\intertext{and on the other hand, using Brillinger's formula from
  Theorem~\ref{thm:ConditionedCumulants} we have}
    C_n^\phi&(X_1 b_1,
             X_2 b_2,
             \dots,
             X_n b_n) \\
    &= \sum_{\pi\in\NC_n}
        C_n^\phi\circ C_\pi^\psi(X_1 b_1,
                                X_2 b_2,
                                \dots,
                                X_n b_n) \\
    &= \phi( C_n^\psi(X_1 b_1, X_2 b_2, \dots, X_n b_n))
       +
       \sum_{\pi<\hat1_n}
        C_n^\phi\circ C_\pi^\psi(X_1 b_1,
                                X_2 b_2,
                                \dots,
                                X_n b_n) 
  \end{align*}
  Comparing the two expressions, it suffices to 
  prove inductively for $\pi<\hat1_n$ the identity
  \begin{equation}
    \label{eq:CphioCpipsi=CpiuKpiphi}
    C_n^\phi\circ C_\pi^\psi(X_1 b_1,
                             X_2 b_2,
                             \dots,
                             X_n b_n)
    = C_{\pi\interw K(\pi)}^\phi(X_1,b_1,
                              X_2, b_2,
                              \dots,
                              X_n, b_n)
   .
  \end{equation}
  Indeed,
  \begin{align*}
  C^\phi_n\circ C_\pi^\psi(X_1 b_1,
                             X_2 b_2,
                             \dots,
                             X_n b_n)
  &=\sum_{\rho\geq\pi} \phi_\rho\circ C^\psi_\pi(X_1 b_1,
                             X_2 b_2,
                             \dots,
                             X_n b_n)
                             \,
                             \mu(\rho,\hat{1}_n)
 \\
\intertext{
  and some $b_i$'s are replaced by $\phi(b_i)$,
  namley those, which are inside a block of $\pi$,
  which means, that they are singletons in $K(\pi)$.
  By induction hypothesis we obtain}
   &=
   \sum_{\rho\geq\pi} \phi_\rho\circ C^\phi_\pi(X_1 \tilde{b}_1,
                             X_2 \tilde{b}_2,
                             \dots,
                             X_n \tilde{b}_n)
                             \,
                             \mu(\rho,\hat{1}_n)
\end{align*}
where
$$
\tilde{b}_i =
\begin{cases}
  \phi(b_i) & \text{if $i$ is a singleton of $K(\pi)$}\\
  b_i       &\text{otherwise}
\end{cases}
$$
``otherwise'' meaning that $i$ is right next to an end point
    of a block of $\pi$, i.e., it is marked with  `$\times$'
    in the proof of Proposition~\ref{prop:kreweras}.
It is now easy to see that this is equal to
\begin{multline*}
    C^\phi_{\pi\interw K(\pi)}(X_1, b_1,
                       X_2, b_2,
                       \dots,
                       X_n, b_n)\\
    = \sum_{\sigma\leq K(\pi)}
      C^\phi_\pi\interw \phi_\sigma(X_1,b_1,
                                    X_2, b_2,
                                    \dots,
                                    X_n, b_n)
        \, \mu(\sigma,K(\pi))
    .
\end{multline*}
where $C^\phi_\pi\interw \phi_\sigma$ denotes the interweaved
product of the cumulant $C^\phi_\pi$ with the partitioned expectation
$\phi_\sigma$.
\end{proof}

\emph{Acknowledgements.}
We are grateful to Roland Speicher for a simplification in the first part
of the proof of Theorem~\ref{thm:CondCum:NSSthm3.6}.


\providecommand{\bysame}{\leavevmode\hbox to3em{\hrulefill}\thinspace}
\providecommand{\MR}{\relax\ifhmode\unskip\space\fi MR }
\providecommand{\MRhref}[2]{%
  \href{http://www.ams.org/mathscinet-getitem?mr=#1}{#2}
}
\providecommand{\href}[2]{#2}

\end{document}